\documentclass[preprint,review,10pt]{amsart}
\usepackage{amsmath,amssymb,amsfonts,xspace}
\newtheorem{theorem}{Theorem}[section]

\theoremstyle{definition}
\newtheorem{definition}[theorem]{Definition}
\newtheorem{example}[theorem]{Example}

\newtheorem{corollary}[theorem]{Corollary}

\theoremstyle{remark}

\numberwithin{equation}{section}

\begin{document}

\title{
$wsq$-primary hyperideals in a Krasner $(m,n)$-hyperring
  }

\author{M. Anbarloei}
\address{Department of Mathematics, Faculty of Sciences,
Imam Khomeini International University, Qazvin, Iran.
}

\email{m.anbarloei@sci.ikiu.ac.ir }


\subjclass[2010]{ 20N20, 16Y99. 
}


\keywords{  $n$-ary $q$-primary hyperideal, $(k,n)$-absorbing $q$-primary hyperideal, $n$-ary $sq$-primary hyperideal, $n$-ary $wsq$-primary hyperideal.}

\begin{abstract}
In this paper, we present a new class of hyperideals: called weakly strongly quasi-primary (briefly, $wsq$-primary) hyperideal. For this purpose we first need to introduce
the notions of quasi-primary and strongly quasi-primary hyperideals. After the definition and investigation of them, we introduce and study weakly strongly quasi-primary hyperideals. A proper hyperideal $P$ of a Krasner $(m,n)$-hyperring $R$ is said to be  $n$-ary weakly strongly quasi-primary if  $0 \neq g(r_1^n) \in P$ for each $r_1^n \in R$ implies that $g(r_i^{(2)},1^{(n-2)}) \in P$ or $ g(r_1^{i-1},1,r_{i+1}^n) \in {\bf r^{(m,n)}}(P)$ for some $1 \leq i \leq n$.  Several properties and characterizations concerning the concept  are presented. The stability of this new concept with respect to various hyperring-theoretic constructions is studied. 

\end{abstract}
\maketitle
\section{Introduction}
 The prime and primary ideals are the remarkably important  structures. A proper ideal of a commutative ring $R$ is called quasi-primary if its radical is prime. This concept was introduced by Fuchs in \cite{Fuchs}. Some operations such as saturation and  idealization on quasi-primary ideals were presented in \cite{Samei}. The notion of 2-absorbing quasi-primary ideals as a generalization of quasi-primary ideals was given in \cite{Tekir}. An intermediate class of primary ideals and quasi primary ideals which is called strongly quasi primary ideals was introduced and investigated by Koc et al. \cite{Koc}. Moreover, they have constructed a subgraph of ideal based on zero divisor graph characterizing strongly quasi primary ideals and   have found when two graphs are equal. In \cite{Ugurlu}, Ugurlu et al. defined and studied the concept of  weakly strongly quasi primary ideals. A proper ideal $I$ of a commutative ring $R$ is called weakly strongly quasi primary if $0\neq xy \in I$ for some $x,y \in R$ implies that $x^2 \in I$ or $y \in \sqrt{I}$. 


Krasner Hyperrings are an  weighty class of algebraic hyperstructures.
In the structure,  the addition is a hyperoperation, while the multiplication is an ordinary binary operation. A generalization of the structure, which is a subclass of $(m,n)$-hyperrings, was defined  in \cite{d1}. It is called Krasner $(m,n)$-hyperring. \cite{d1} $(R, f, g)$, or simply $R$, is called a Krasner $(m, n)$-hyperring if:
(1) $(R, f$) is a canonical $m$-ary hypergroup;
(2) $(R, g)$ is a $n$-ary semigroup;
(3) The $n$-ary operation $g$ is distributive with respect to the $m$-ary hyperoperation $f$ , i.e., for every $a^{i-1}_1 , a^n_{ i+1}, x^m_ 1 \in R$, and $1 \leq i \leq n$,
\[g\bigg(a^{i-1}_1, f(x^m _1 ), a^n _{i+1}\bigg) = f\bigg(g(a^{i-1}_1, x_1, a^n_{ i+1}),..., g(a^{i-1}_1, x_m, a^n_{ i+1})\bigg);\]
(4) $0$ is a zero element (absorbing element) of the $n$-ary operation $g$, i.e., for every $x^n_ 2 \in R$ , 
$g(0, a^n _2) = g(a_2, 0, x^n _3) = ... = g(a^n_ 2, 0) = 0$. 
A non-empty subset $S$ of $R$ is called a subhyperring of $R$ if $(S, f, g)$ is a Krasner $(m, n)$-hyperring. 
The non-empty subset $I$ of $(R,f,g)$  is a hyperideal  if $(I, f)$ is an $m$-ary subhypergroup of $(R, f)$ and $g(a^{i-1}_1, I, a_{i+1}^n) \subseteq I$, for every $a^n _1 \in  R$ and  $1 \leq i \leq n$. 
Note that  $a^j_i$ denotes the sequence $a_i, a_{i+1},..., a_j$.  $a^j_i$ is the empty symbol if $j< i$. Using this notation,
$f(a_1,..., a_i, b_{i+1},..., b_j, c_{j+1},..., c_n)$
will be written as $f(a^i_1, b^j_{i+1},c^n_{j+1})$. The  expression will be written in the form $f(a^i_1, b^{(j-i)}, c^n_{j+1}),$ where $b_{i+1} =... = b_j = b$.  
For non-empty subsets $H_1^n$ of $R$, define
$f(H^n_1) = \bigcup \{f(a^n_1) \ \vert \  a_i \in  H_i, 1 \leq i \leq n \}.$
Some important hyperideals such as nilradical, Jacobson radical, $n$-ary prime and primary hyperideals  of Krasner $(m, n)$-hyperrings  were introduced in \cite{sorc1}. A hyperideal $M$ of  $R$
is said to be maximal if for every hyperideal $N$ of $R$, $M \subseteq N \subseteq R$ implies that $N=M$ or $N=R$.
The Jacobson radical of a Krasner $(m, n)$-hyperring $R$
is the intersection of all maximal hyperideals of $R$ and it is denoted by $J_{(m,n)}(R)$. If $R$ does not have any maximal hyperideal, we let $J_{(m,n)}(R)=R$. A proper hyperideal $P$ of a Krasner $(m, n)$-hyperring $R$
  is called prime  if  $g(A_1^ n) \subseteq P$ for hyperideals $A_1^n$ of $R$ implies that $A_1 \subseteq P$ or $A_2 \subseteq P$ or $\cdots$ or $A_n \subseteq P$. By Lemma 4.5 in \cite{sorc1}, a proper hyperideal $P$ of a Krasner $(m, n)$-hyperring $R$ is  prime  if for all $a^n_ 1 \in R$, $g(a^n_ 1) \in P$ implies that $a_1 \in P$ or $\cdots$ or $a_n \in P$. 
 Let $I$ be a hyperideal in a  Krasner $(m, n)$-hyperring $R$ with
scalar identity. The radical  of $I$, denoted by ${\bf r^{(m,n)}}(I)$ is  the intersection is taken over all  prime hyperideals $P$ which contain $I$. If the set of all prime hyperideals containing $I$ is empty, then ${\bf r^{(m,n)}}(I)=R$.  It was shown that if $a \in {\bf r^{(m,n)}}(I)$, then 
 there exists $s \in \mathbb {N}$ such that $g(a^ {(s)} , 1_R^{(n-s)} ) \in I$ for $s \leq n$, or $g_{(l)} (a^ {(s)} ) \in I$ for $s = l(n-1) + 1$.
 A proper hyperideal $I$ of a  Krasner $(m, n)$-hyperring $R$ with the scalar identity $1_R$  is said to be a   primary hyperideal if $g(a^n _1) \in I$ and $a_i \notin I$ implies that $g(a_1^{i-1}, 1_R, a_{ i+1}^n) \in {\bf r^{(m,n)}}(I)$ for some $1 \leq i \leq n$. By Theorem 4.28 in \cite{sorc1}, ${\bf r^{(m,n)}}(I)$ is  a prime hyperideal of $R$ if $I$ is a primary hyperideal in a  Krasner $(m, n)$-hyperring $R$ with the scalar identity $1_R$.

The concept of $(k,n)$-absorbing (primary) hyperideals was studied by Hila et al. \cite{rev2}. 
Norouzi et al.  presented a new definition for normal hyperideals in Krasner  \cite{nour}.  Asadi and Ameri  studied direct limit of a direct system in the category of Krasner $(m,n)$-hyperrigs \cite{asadi}. 
Dongsheng  defined the notion of $\delta$-primary ideals in  a commutative ring  where  $\delta$ is a function that assigns to each ideal $I$  an ideal $\delta(I)$ of the same ring \cite{bmb2}.  Also, he and his colleague  proposed the notion of 2-absorbing $\delta$-primary ideals  which  unifies 2-absorbing ideals and 2-absorbing primary ideals in \cite{bmb3}. Ozel Ay et al.  extended the notion of $\delta$-primary  on Krasner hyperrings \cite{bmb4}. The notion of $\delta$-primary  hyperideals in Krasner $(m,n)$-hyperrings, which unifies the prime and primary hyperideals under one frame, was introduced in \cite{mah3}. 

In this paper, after presenting two classes of hyperideals in a Krasner $(m,n)$-hyperring, we introduce the notion of weakly strongly quasi-primary hyperideals. Throughout this article, we focus only on commutative Krasner $(m,n)$-hyperrings with a nonzero identity $1$. $R$ will be a commutative Krasner $(m,n)$-hyperring. The paper is orgnized as follows. In Section 2, we first define the concept of $n$-ary quasi-primary (briefly, $q$-primary) hyperideals of $R$ and then introduce the notion of $(k,n)$-absorbing quasi-primary as a generalization of the quasi-primary hyperideals. After the definition of the $(k,n)$-absorbing quasi-primary (briefly, $(k,n)$-absorbing $q$-primary) hyperideals, their chief properties will be shown. Section 3 is devoted for studing the notion of  strongly quasi-primary (briefly, $sq$-primary) hyperideals. In Section 4, we introduce weakly strongly quasi-primary (briefly, $wsq$-primary) hyperideals. The stability of this  notion with respect to various hyperring-theoretic constructions is studied.  Section 5, concerns the conclusion. 
\section{$n$-ary $q$-primary hyperideals}
In this section, we first study the notion of $n$-ary quasi-primary hyperideals of $R$ and then we extend  the concept to the notion of  $(k,n)$-absorbing quasi-primary. After the definition of the $(k,n)$-absorbing quasi-primary hyperideals, their  properties will be given.
\begin{definition}
 A proper hyperideal $P$ of $R$ is called  $n$-ary quasi-primary  (briefly, $q$-primary) provided that ${\bf r^{(m,n)}}(P)$ is an $n$-ary prime hyperideal of $R$.
\end{definition}
\begin{example} \label{classes} 
Suppose that $\mathbb{Z}_{12}=\{0,1,2,3,\cdots,11\}$ is the set of all congruence classes of integers modulo $12$ and $\mathbb{Z}^{\star}_{12}=\{1,5,7,11\}$ is multiplicative subgroup of units $\mathbb{Z}_{12}$. Construct $G$ as $\mathbb{Z}_{12}/\mathbb{Z}_{12}^{\star}$. Then we have $G=\{\bar{0},\bar{1},\bar{2},\bar{3},\bar{4},\bar{6}\}$ in which $\bar{0}=\{0\}$, $\bar{1}=\{1,5,7,11\}$, $\bar{2}=\bar{10}=\{2,10\}$, $\bar{3}=\bar{9}=\{3,9\}$, $\bar{4}=\bar{8}=\{4,8\}$, $\bar{6}=\{6\}$. Consider Krasner hyperring $(G,\boxplus,\circ)$ that for all $\bar{a},\bar{b} \in G$, $\bar{a} \circ \bar{b}=\overline{ab}$ and 2-ary hyperoperation $\boxplus$ is defined as follows:

\[\begin{tabular}{|c|c|c|c|c|c|c|} 
\hline $\boxplus$ & $\bar{0}$ & $\bar{1}$ & $\bar{2}$ & $\bar{3}$ & $\bar{4}$ & $\bar{6}$
\\ \hline $\bar{0}$ & $\bar{0}$ & $\bar{1}$ & $\bar{2}$ & $\bar{3}$& $\bar{4}$ & $\bar{6}$
\\ \hline $\bar{1}$ & $\bar{1}$ & $\bar{0},\bar{2},\bar{\bar{4}},\bar{6}$ & $\bar{1},\bar{3}$ & $\bar{2},\bar{4}$ & $\bar{1},\bar{3}$ & $\bar{1}$
\\ \hline $\bar{2}$ & $\bar{2}$ & $\bar{1},\bar{3}$ & $\bar{0},\bar{4}$ & $\bar{1}$ & $\bar{2},\bar{6}$ & $\bar{4}$
\\ \hline $\bar{3}$ & $\bar{3}$ & $\bar{2},\bar{4}$ & $\bar{1}$ & $\bar{0},\bar{6}$ & $\bar{1}$ & $\bar{3}$
\\ \hline $\bar{4}$ & $\bar{4}$ & $\bar{1},\bar{3}$ & $\bar{2},\bar{6}$ &  $\bar{1}$ & $\bar{0},\bar{4}$ & $\bar{2}$
\\ \hline $\bar{6}$ & $\bar{6}$ & $\bar{1}$ & $\bar{4}$ & $\bar{3}$ & $\bar{2}$ & $\bar{0}$
\\ \hline
\end{tabular}\]

$\vspace{0.3cm}$

Consider the hyperideal $P=\{\bar{0},\bar{4}\}$ of $G$. Then we have ${\bf r^{(2,2)}}(P)=\{\bar{0},\bar{2},\bar{4},\bar{6},\}$. It is easy to see that the radical of the hyperideal $P$ is prime and so $P$ is  $q$-primary. 
\end{example}




  

Let $(R_1, f_1, g_1)$ and $(R_2, f_2, g_2)$ be two Krasner $(m,n)$-hyperrings such that $1_{R_1}$ and $1_{R_2}$ be scalar identitis of $R_1$ and $R_2$, respectively. Then 
the $(m, n)$-hyperring $(R_1 \times R_2, f_1\times f_2 ,g_1 \times g_2 )$ is defined by m-ary hyperoperation
$f=f_1\times f_2 $ and n-ary operation $g=g_1 \times g_2$, as follows:

$f_1 \times f_2((a_{1}, b_{1}),\cdots,(a_m,b_m)) = \{(a,b) \ \vert \ \ a \in f_1(a_1^m), b \in f_2(b_1^m) \}$

$g_1 \times g_2 ((x_1,y_1),\cdots,(x_n,y_n)) =(g_1(x_1^n),g_2(y_1^n)) $,\\
for all $a_1^m,x_1^n \in R_1$ and $b_1^m,y_1^n \in R_2$ \cite{mah2}. 
\begin{theorem}
Suppose that  $(R_1, f_1, g_1)$ and $(R_2, f_2, g_2)$ are two Krasner $(m,n)$-hyperrings such that $1_{R_1}$ and $1_{R_2}$ be scalar identitis of $R_1$ and $R_2$, respectively. Assume that $P$ is a proper hyperideal of $R_1 \times R_2$. Then $P$ is an $n$-ary $q$-primary hyperideal of $R_1 \times R_2$ if and only if $P=P_1 \times R_2$ for some  $n$-ary $q$-primary hyperideal $P_1$ of $R_1$ or $P=R_1 \times P_2$ for some  $n$-ary $q$-primary hyperideal $P_2$ of $R_2$.
\end{theorem}
\begin{proof}
$\Longrightarrow$ Let $P$ be an $n$-ary $q$-primary hyperideal of $R_1 \times R_2$. Since $P$ is a proper hyperideal of $R_1 \times R_2$, then there exist some hyperideals $P_1$ and $P_2$ of $R_1$ and $R_2$, respectively, such that  $P=P_1 \times P_2$.  Since $P$ is $n$-ary $q$-primary, then ${\bf r^{(m,n)}}(P)={\bf r^{(m,n)}}(P_1) \times {\bf r^{(m,n)}}(P_2)$ is an $n$-ary prime hyperideal of $R_1 \times R_2$. This implies that ${\bf r^{(m,n)}}(P_1)$ is an $n$-ary prime hyperideal of $R_1$ and ${\bf r^{(m,n)}}(P_2)=R_2$ or ${\bf r^{(m,n)}}(P_2)$ is an $n$-ary prime hyperideal of $R_2$ and ${\bf r^{(m,n)}}(P_1)=R_1$ which means $P=P_1 \times R_2$ or $P=R_1 \times P_2$ for some $n$-ary $q$-primary hyperideals $P_1$ and $P_2$ of $R_1$ and $R_2$, respectively.\\
$\Longleftarrow$ It is straightforward.
\end{proof}
We give the following results obtained by the previous theorem.
\begin{corollary}
Let $(R_i, f_i, g_i)$ be a Krasner $(m,n)$-hyperring for each $1 \leq i \leq t$ such that $1_{R_i}$ is scalar identity of $R_i$. Assume that $P$ is a proper hyperideal of $R_1 \times \cdots \times R_t$. Then $P$ is an $n$-ary $q$-primary hyperideal of $R_1 \times \cdots \times R_t$ if and only if $P=P_1 \times P_2 \times \cdots \times P_t$ such that $P_u$ is an  $n$-ary $q$-primary hyperideal of $R_u$ for some $1 \leq u \leq t$ and $P_k=R_k$ for all $1 \leq k \leq t$ such that $k \neq u$.
\end{corollary}
In his paper \cite{rev2}, Hila et al. introduced a generalization of $n$-ary prime hyperideals of Krasner $(m,n)$-hyperrings, which they defined as $(k,n)$-absorbing hyperideals. Let $k$ be a positive integer. A proper hyperideal $P$ of $R$ is said to be  $(k,n)$-absorbing if whenever $g(r_1^{kn-k+1}) \in P$ for $r_1^{kn-k+1} \in R$, then there are $(k-1)n-k+2$ of the $r_i^,$s whose $g$-product is in $P$. Moreover, they generalized this concept to the notion of $(k,n)$-absorbing primary hyperideals. A proper hyperideal $P$ of $R$ is called $(k,n)$-absorbing primary if whenever $g(r_1^{kn-k+1}) \in P$ for some $r_1^{kn-k+1} \in R$, then $g(r_1^{(k-1)n-k+2)}) \in P$ or a $g$-product of $(k-1)n-k+2$ of the $r_i^,$s except $g(r_1^{(k-1)n-k+2})$ is in ${\bf r^{(m,n)}}(P)$. Theorem 4.6 in \cite{rev2} shows that the radical of an $(k,n)$-absorbing primary hyperideal is an $(k,n)$-absorbing hyperideal of $R$. Now, we aim to study hyperiseals whose radical is an $(k,n)$-absorbing hyperideal of $R$.
\begin{definition}
 A proper hyperideal $P$ of $R$ is called  $(k,n)$-absorbing quasi-primary  (briefly, $(k,n)$-absorbing $q$-primary) if ${\bf r^{(m,n)}}(P)$ is an $(k,n)$-absorbing hyperideal of $R$.
\end{definition}
If $P$ is an $(k,n)$-absorbing $q$-primary hyperideal of $R$ such that ${\bf r^{(m,n)}}(P)=Q$, then we say that $P$ is an $Q$-$(k,n)$-absorbing $q$-primary hyperideal of $R$.
\begin{example} 
Consider the Krasner $(2,3)$-hyperring $(K=[0,1],+, \cdot)$ such that $"\cdot"$ is the usual multiplication on real numbers and  2-ary hyperoperation $"+"$ is defined as follows:
\[
a + b=
\begin{cases}
\{max\{a, b\}\}, & \text{if $a \neq b$}\\
$[0,a]$, & \text{if $a =b.$}
\end{cases}\]

$\vspace{0.2cm}$

Then the hyperideal $S=[0,0.5]$ is a $(2,2)$-absorbing   $q$-primary hyperideal of $K$.
\end{example}
\begin{theorem}
$(1)$ Every $n$-ary $q$-primary hyperideal of $R$ is $(2,n)$-absorbing $q$-primary.

$(2)$ Every $(k,n)$-absorbing primary hyperideal of $R$ is $(k,n)$-absorbing $q$-primary.
\end{theorem}
\begin{proof}
$(1)$ Let $P$ be an $n$-ary $q$-primary hyperideal of $R$. This means that ${\bf r^{(m,n)}}(P)$ is an $n$-ary prime hyperideal of $R$. So ${\bf r^{(m,n)}}(P)$ is an $(2,n)$-absorbing hyperideal of $R$. Thus $P$ is $(2,n)$-absorbing $q$-primary of $R$.

$(2)$ Let $P$ be an $(k,n)$-absorbing primary hyperideal of $R$. Then ${\bf r^{(m,n)}}(P)$ is an $(k,n)$-absorbing hyperideal of $R$ by Theorem 4.6 in \cite{rev2}. This means that $P$ is an $(k,n)$-absorbing $q$-primary hyperideal of $R$.
\end{proof}
\begin{theorem}
Let $P_1^t$ be $Q$-$(k,n)$-absorbing $q$-primary hyperideals of $R$ for some $(k,n)$-absorbing hyperideal $Q$ of $R$. Then $P=\cap_{i=1}^t P_i$ is a $Q$-$(k,n)$-absorbing $q$-primary hyperideal of $R$.
\end{theorem}
\begin{proof}
Assume that $Q$ is an $(k,n)$-absorbing hyperideal of $R$ and $P_1^t$ are $Q$-$(k,n)$-absorbing $q$-primary hyperideals of $R$. We have ${\bf r^{(m,n)}}(P)={\bf r^{(m,n)}}(\cap_{i=1}^t P_i)=\cap_{i=1}^t{\bf r^{(m,n)}}(P_i)=Q$ which shows $P$ is a $Q$-$(k,n)$-absorbing $q$-primary hyperideal of $R$.
\end{proof}
\begin{theorem}
Let $P$ be a proper hyperideal of $R$. Then $P$ is an $(k,n)$-absorbing $q$-primary hyperideal of $R$ if and only if    $g(r_1^{kn-k+1}) \in P$ for $r_1^{kn-k+1} \in R$ implies that  there exist $(k-1)n-k+2$ of the $r_i^,$s whose $g$-product is in ${\bf r^{(m,n)}}(P)$.
\end{theorem}
\begin{proof}
$\Longrightarrow$ Let   $g(r_1^{kn-k+1}) \in P$ for some $r_1^{kn-k+1} \in R$. Then $g(r_1^{kn-k+1}) \in {\bf r^{(m,n)}}(P)$. Since $P$ is an $(k,n)$-absorbing $q$-primary hyperideal of $R$, then ${\bf r^{(m,n)}}(P)$ is an $(k,n)$-absorbing hyperideal of $R$. Therefore there exist $(k-1)n-k+2$ of the $r_i^,$s whose $g$-product is in ${\bf r^{(m,n)}}(P)$.

$\Longleftarrow$ Assume that  $g(r_1^{kn-k+1}) \in {\bf r^{(m,n)}}(P)$ for some $r_1^{kn-k+1} \in R$ such that all products of $(k-1)n-k+2$ of the $r_i^,$s, other than  $g(r_1^{(k-1)n-k+2})$, are not in ${\bf r^{(m,n)}}(P)$. Since $g(r_1^{kn-k+1}) \in {\bf r^{(m,n)}}(P)$, then there exists $s \in \mathbb{N}$ such that if $s \leq n$, then $g(g(r_1^{kn-k+1})^{(s)},1^{(n-s)}) \in P$ and if $s>n$, $s=l(n-1)+1$, then $g_{(l)}(g(r_1^{kn-k+1})^{(s)}) \in P$. In the former case, we get $g(g(r_1)^{(s)},\cdots,g(r_{kn-k+1})^{(s)},1^{(n-s)}) \in P$. By the assumption, we have 

$g(g(r_1)^{(s)},\cdots,g(r_{(k-1)n-k+2})^{(s)},1^{(n-s)})=g(g(r_1^{(k-1)n-k+2})^{(s)},1^{(n-s)}) \in {\bf r^{(m,n)}}(P)$\\ which means $ g(r_1^{(k-1)n-k+2}) \in {\bf r^{(m,n)}}(P)$. This shows that $P$ is an $(k,n)$-absorbing $q$-primary hyperideal of $R$. By a similar argument, we can prove the claim for the other case.
\end{proof}

\begin{theorem} 
Let $P$ be an $(k,n)$-absorbing $q$-primary hyperideal of $R$. Then $P$ is an  $(u,n)$-absorbing $q$-primary hyperideal of $R$ for all $u>n$. 
\end{theorem}
\begin{proof}
Let $P$ be an $(k,n)$-absorbing $q$-primary hyperideal of $R$. Then ${\bf r^{(m,n)}}(P)$ is an $(k,n)$-absorbing hyperideal of $R$. By Theorem 3.7 in \cite{rev2}, we conclude that ${\bf r^{(m,n)}}(P)$ is an $(u,n)$-absorbing hyperideal of $R$ for all $u>n$. This means that  $P$ is an $(u,n)$-absorbing $q$-primary hyperideal of $R$ for all $u>n$.
\end{proof}
\section{$n$-ary $sq$-primary hyperideals}
 Our aim in  this section is to define and study the notion of  strongly quasi-primary  hyperideals. Indeed, these hyperideals are an intermediate class of primary hyperideals and $q$-primary hyperideals.

\begin{definition}
Let $P$ be a proper hyperideal of $R$. $P$ refers to an $n$-ary strongly quasi primary (briefly, $sq$-primary) hyperideal if $r_1^n \in R$ and $g(r_1^n) \in P$ imply $g(r_i^{(2)},1^{(n-2)}) \in P$ or $g(r_1^{i-1},1,r_{i+1}^n) \in {\bf r^{(m,n)}}(P)$ for some $1 \leq i \leq n$.
\end{definition}
\begin{example}
 
Consider the  Krasner $(3,3)$-hyperring $(H=\{0,1,2\},f,g)$ such that    3-ary hyperoeration $f$ and 3-ary operation $g$ are defined as follow:
\[f(0^{(3)})=0, \ \ \ f(0^{(2)},1)=1, \ \ \ f(0,1^{(2)})=1, \ \ \ f(1^{(3)})=1, \ \ \ f(1^{(2)},2)=H\]
\[f(0,1,2)=H, \ \ \ f(0^{(2)},2)=2,\ \ \ f(0,a^{(2)})=2,\ \ \ f(1,2^{(2)})=H, \ \ \ f(2^{(3)})=2\]
$\ \ \ g(1^{(3)})=1,\ \ \ \ g(1^{(2)},2)=g(1,2^{(2)})=g(2^{(3)})=2$\\

and for $a_1^2 \in H, g(0,a_1^2)=0$.
Then the hyperideal  $T=\{0,2\}$  is a 3-ary $sq$-primary hyperideal of $H$. 
\end{example}
\begin{theorem} \label{ali}
Let $P$ be a proper hyperideal of $R$. If $P$ is an $n$-ary $sq$-primary hyperideal of $R$, then $P$ is an $n$-ary $q$-primary hyperideal of $R$.
\end{theorem}
\begin{proof}
Let $P$ be an $n$-ary $sq$-primary hyperideal of $R$. Assume that $g(r_1^n) \in {\bf r^{(m,n)}}(P)$ for some $r_1^n \in R$. Then there exists $s \in \mathbb{N}$ such that if $s \leq n$, then $g(g(r_1^n)^{(s)},1^{(n-s)}) \in P$. By
associativity we get
 
  $\hspace{1.5cm}g(r_i^{(s)},g(r_1^{i-1},1,r_{i+1}^n)^{(s)},1^{(n-2s)})$

$\hspace{2cm}=g(r_i^{(s)},g(r_1^{i-1},1,r_{i+1}^n)^{(s)},g(1^{(n)}),1^{(n-2s-1)})$

$\hspace*{2cm}=g(g(r_i^{(s)},1^{(n-s)}),g(g(r_1^{i-1},1,r_{i+1}^n)^{(s)},1^{(n-s)}),1^{(n-2)}) $

$\hspace{2cm}\in P$.\\
Since $P$ is $sq$-primary,  then $g(g(r_i^{(s)},1^{(n-s)})^{(2)},1^{(n-2)})=g(r_i^{(2s)},1^{(n-2s)}) \in P$ or $g(g(r_1^{i-1},1,r_{i+1}^n)^{(s)},1^{(n-s)}),1^{(n-1)})=g(g(r_1^{i-1},1,r_{i+1}^n)^{(s)},1^{(n-s)}) \in {\bf r^{(m,n)}}(P)$ for some $1 \leq i \leq n$.  This implies that  $r_i \in  {\bf r^{(m,n)}}(P)$ or    $g(r_1^{i-1},1,r_{i+1}^n) \in {\bf r^{(m,n)}}(P)$. In the first possibility, we are done. In the second possibility, we can
continue the process and obtain $r_j \in {\bf r^{(m,n)}}(P)$ for some $1 \leq j \leq i-1$ or $i+1 \leq j \leq n$. If $s=l(n-1)+1$,  then we are done  similarly.  Hence ${\bf r^{(m,n)}}(P)$ is an $n$-ary  prime hyperideal of $R$. Thus $P$ is an $n$-ary $q$-primary hyperideal of $R$.
\end{proof}
Now, we determine when an $n$-ary $q$-primary hyperideal of $R$ is  an $n$-ary $sq$-primary hyperideal of $R$.
\begin{theorem}
Let $P$ be a proper hyperideal of $R$ such that $g({\bf r^{(m,n)}}(P)^{(2)},1^{(n-2)}) \subseteq P$. If $P$ is an $n$-ary $q$-primary hyperideal of $R$, then $P$ is an $n$-ary $sq$-primary hyperideal of $R$.
\end{theorem}
\begin{proof}
Let $P$ be an $n$-ary $q$-primary hyperideal of $R$. Assume that $g(r_1^n) \in P$ for some $r_1^n \in R$. Since $P \subseteq {\bf r^{(m,n)}}(P)$ and $P$ is an $n$-ary $q$-primary hyperideal of $R$, we get $r_i \in {\bf r^{(m,n)}}(P)$ for some $1 \leq i \leq n$. Since $g({\bf r^{(m,n)}}(P)^{(2)},1^{(n-2)}) \subseteq P$, we have $g(r_i^{(2)},1^{(n-2)}) \in P$ which means $P$ is an $n$-ary $sq$-primary hyperideal of $R$.
\end{proof}
Recall from \cite{sorc1} that  the hyperideal generated by  an element $x$ in a Krasner $(m,n)$-hyperring $R$ is denoted by $<x>$ and is defined as $<x>=g(R,x,1^{(n-2)})=\{g(r,x,1^{(n-2)}) \ \vert \ r \in R\}.$
\begin{theorem}
Let $<r>$ be an $n$-ary $sq$-primary hyperideal of $R$ for all $r \in R$. Then every proper hyperideal of $R$ is an $n$-ary $sq$-primary hyperideal.
\end{theorem}
\begin{proof}
Let $P$ be an arbitrary hyperideal of $R$. Assume that $g(r_1^n) \in P$ for some $r_1^n \in R$. Then we have $g(r_1^n) \in <g(r_1^n)> $. Since $<g(r_1^n)> $ is an $n$-ary $sq$-primary hyperideal of $R$, we obtain $g(r_i^{(2)},1^{(n-2)}) \in <g(r_1^n)> $ or $g(r_1^{i-1},1,r_{i+1}^n) \in {\bf r^{(m,n)}}(<g(r_1^n)>)$ for some $1 \leq i \leq n$. Since $<g(r_1^n)> \subseteq P$ and ${\bf r^{(m,n)}}(<g(r_1^n)>) \subseteq {\bf r^{(m,n)}}(P)$, we get $g(r_i^{(2)},1^{(n-2)}) \in P$ or $g(r_1^{i-1},1,r_{i+1}^n) \in {\bf r^{(m,n)}}(P)$. Thus $P$ is an $n$-ary $sq$-primary hyperideal of $R$. 
\end{proof}
\begin{theorem}
Let $P$ be a proper hyperideal of $R$. If $P$ is an $n$-ary $sq$-primary hyperideal of $R$, then $g(P_1^n) \subseteq P$ for some hyperideals $P_1^n$ of $R$ implies that $g(p^{(2)},1^{(n-2)}) \in P$ for all $p \in P_i$ or $g(P_1^{i-1},1,P_{i+1}^n) \subseteq {\bf r^{(m,n)}}(P)$ for some $1 \leq i \leq n$.
\end{theorem}
\begin{proof}
Assume that $P$ is an $n$-ary $sq$-primary hyperideal of $R$. First of all,  we prove that $ < r> \subseteq P_r$ or $P_r \subseteq {\bf r^{(m,n)}}(P)$ for all $r \in R$ where $P_r=\{a \in R \ \vert \ g(r,a,1^{(n-2)}) \in P\}$. To establish the claim, we pick an element $r \in R$. If $g(r^{(2)},1^{(n-2)}) \in P$, then we get $ < r> \subseteq P_r$. We assume that $g(r^{(2)},1^{(n-2)}) \notin P$. Take any $a \in P_r$. So we have $g(r,a,1^{(n-2)}) \in P$. Since $P$ is an $n$-ary $sq$-primary hyperideal of $R$ and $g(r^{(2)},1^{(n-2)}) \notin P$, we conclude that $a=g(a,1^{(n-1)}) \in {\bf r^{(m,n)}}(P)$, as needed. Now, we assume that $g(P_1^n) \subseteq P$ for some hyperideals $P_1^n$ of $R$ such that $g(P_1^{i-1},1,P_{i+1}^n) \nsubseteq {\bf r^{(m,n)}}(P)$ for some $1 \leq i \leq n$. Therefore   there exist $p_j \in P_j$ for each $j \in \{1,\cdots,n\}-\{i\}$ such that $g(p_1^{i-1},1,p_{i+1}^n) \nsubseteq {\bf r^{(m,n)}}(P)$. Hence for every $p \in P_i$, $g(p_1^{i-1},p,p_{i+1}^n) \in  {\bf r^{(m,n)}}(P)$. Since $g(p_1^{i-1},1,p_{i+1}^n) \in P_p-{\bf r^{(m,n)}}(P)$, then $P_r \nsubseteq {\bf r^{(m,n)}}(P)$ and so $ <p> \subseteq P_p$. This implies that $g(p^{(2)},1^{(n-2)}) \in P$.
\end{proof}
\begin{theorem}
Let $P$ be an $n$-ary $sq$-primary hyperideal of $R$. If $r \notin P$ and  $<r>=<g(r^{(2)},1^{(n-2)})>$, then $P_r=\{a \in R \ \vert \ g(r,a,1^{(n-2)}) \in P \}$  is  an $n$-ary $sq$-primary hyperideal of $R$.
\end{theorem}
\begin{proof}
Assume $P$ is an $n$-ary $sq$-primary hyperideal of $R$. Since $r \notin P_r$ , then $ < r> \nsubseteq P_r$ and so ${\bf r^{(m,n)}}(P)={\bf r^{(m,n)}}(P_r)$. Let $g(r_1^n) \in P_r$ for some $r_1^n \in R$ such that $g(r_1^{i-1},1,r_{i+1}^n) \notin {\bf r^{(m,n)}}(P_r)$. Therefore we have $g(g(r_1^n),r,1^{(n-2)})=g(r_1^{i-1},g(r,r_i,1^{(n-2)}),r_{i+1}^n) \in P$ and $g(r_1^{i-1},1,r_{i+1}^n) \notin {\bf r^{(m,n)}}(P)$. Then we conclude that $g(g(r,r_i,1^{(n-2)})^{(2)},1^{(n-2)})=g(g(r^{(2)},1^{(n-2)}),g(r_i^{(2)},1^{(n-2)}),1^{(n-2)}) \in P$ as $P$ is an $n$-ary $sq$-primary hyperideal of $R$. This implies that $g(r_i^{(2)},1^{(n-2)}) \in P_{g(r^{(2)},1^{(n-2)})}=P_r$ which means $P_r$ is  an $n$-ary $sq$-primary hyperideal of $R$.
\end{proof}
\begin{theorem} \label{ali2}
Let  $(R_t, f_t, g_t)^,$s be  Krasner $(m,n)$-hyperrings with scalar identitis  $1_{R_t}$ for $t=1,2$ and  $P_t^,$s are hyperideals of $R_t$. Then the followings are equivalent:
\begin{itemize}
\item[\rm(1)]~$P_1 \times P_2$ is an $n$-ary $sq$-primary hyperideal of $R_1 \times R_2$.
 \item[\rm(2)]~$P_1$ is an  $n$-ary $sq$-primary hyperideal   of $R_1$ and $P_2=R_2$ or $P_2$ is an  $n$-ary $sq$-primary hyperideal of  $R_2$ and $P_1=R_1$.
 \end{itemize}
\end{theorem}
\begin{proof}
(1) $\Longrightarrow$ (2) Let $P_1 \times P_2$ be an $n$-ary $sq$-primary hyperideal of $R_1 \times R_2$. Then we have ${\bf r^{(m,n)}}(P_1 \times P_2)={\bf r^{(m,n)}}(P_1) \times {\bf r^{(m,n)}}(P_2)$ is an $n$-ary prime hyperideal of $R_1 \times R_2$ by Theorem \ref{ali} which means $P_1=R_1$ or $P_2=R_1$. We may assume that $P_1=R_1$. Suppose that $g_2(r_1^n) \in P_2$ for some $r_1^n \in R_2$. Therefore we have $g_1 \times g_2((1_{R_1},r_1),\cdots,(1_{R_1},r_n))=(1_{R_1},g_2(r_1^n)) \in P_1 \times P_2$. Since $P_1 \times P_2$ is an $n$-ary $sq$-primary hyperideal of $R_1 \times R_2$, we conclude that $g_1 \times g_2((1_{R_1},r_i)^{(2)},(1_{R_1},1_{R_2})^{(n-2)})=(1_{R_1},g_2(r_i^{(2)},1_{R_2}^{(n-2)})) \in P_1 \times P_2$ or 

$\hspace{1cm}g_1 \times g_2((1_{R_1},r_1),\cdots,(1_{R_1},r_{i-1}),(1_{R_1},1_{R_2}),(1_{R_1},r_{i+1}), \cdots,(1_{R_1},r_n))$   

$\hspace{2cm}=(1_{R_1},g_2(r_1^{i-1},1_{R_2},r_{i+1}^n)) $  

$\hspace{2cm}\in {\bf r^{(m,n)}}(P_1 \times P_2)$\\ for some $1 \leq i \leq n$. This means that $g_2(r_i^{(2)},1_{R_2}^{(n-2)}) \in P_2$ or $g_2(r_1^{i-1},1_{R_2},r_{i+1}^n) \in {\bf r^{(m,n)}}(P_2)$. This shows that $P_2$ is an  $n$-ary $sq$-primary hyperideal of  $R_2$.

(2) $\Longrightarrow$ (1) Let $P_1$ be an  $n$-ary $sq$-primary hyperideal of  $R_1$ and $P_2=R_2$. Assume that $g_1 \times g_2((r_1,s_1),\cdots,(r_n,s_n))=(g_1(r_1^n),g_2(s_1^n)) \in P_1 \times P_2$ for some $r_1^n \in R_1$ and $s_1^n \in R_2$. Hence we have $g_1(r_1^n) \in P_1$. Since $P_1$ is an  $n$-ary $sq$-primary hyperideal of  $R_1$, then $g_1(r_i^{(2)},1_{R_1}^{(n-2)}) \in P_1$ or $g_1(r_1^{i-1},1_{R_1},r_{i+1}^n) \in {\bf r^{(m,n)}}(P_1)$ for some $1 \leq i \leq n$. This means that $g_1 \times g_2((r_i,s_i)^{(2)},(1_{R_1},1_{R_2})^{(n-2)})=(g_1(r_i^{(2)},1_{R_1}^{(n-2)}),g_2(s_i^{(2)},1_{R_2}^{(n-2)})) \in P_1 \times P_2$ or 

$\hspace{1cm}g_1 \times g_2((r_1,s_1),\cdots,(r_{i-1},s_{i-1}),(1_{R_1},1_{R_2}),(r_{i+1},s_{i+1}),\cdots,(r_n,s_n))$

$\hspace{2cm}=(g_1(r_1^{i-1},1_{R_1},r_{i+1}^n),g_1(s_1^{i-1},1_{R_2},s_{i+1}^n))$

$\hspace{2cm} \in {\bf r^{(m,n)}}(P_1\times P_2)$.\\
Similiar for the other case. Thus $P_1 \times P_2$ is an $n$-ary $sq$-primary hyperideal of $R_1 \times R_2$.
\end{proof}
\section{$n$-ary $wsq$-primary hyperideals}
This section  is devoted for studing the notion of $n$-ary weakly strongly quasi-primary hyperideals.
\begin{definition}
Let $P$ be a proper hyperideal of $R$. We call $P$ an $n$-ary weakly strongly quasi-primary hyperideal  of $R$ if  $0 \neq g(r_1^n) \in P$ for each $r_1^n \in R$ implies  $g(r_i^{(2)},1^{(n-2)}) \in P$ or $ g(r_1^{i-1},1,r_{i+1}^n) \in {\bf r^{(m,n)}}(P)$ for some $1 \leq i \leq n$. "weakly strongly quasi-primary" is denoted by "$wsq$-primary", shortly.
\end{definition}
\begin{example} 
Every $n$-ary $sq$-primary hyperideal of $R$ is an $n$-ary $wsq$-primary hyperideal. 
\end{example}
\begin{example} \label{classes}
If we continue with Example \ref{classes}, then $I=\{\bar{0}, \bar{3}, \bar{6}\}$ is a  $wsq$-primary hyperideal  of $G$.
\end{example}

\begin{theorem}\label{radic}
 Assume that $P$ is an $n$-ary $wsq$-primary hyperideal of $R$. If $P$ is not $sq$-primary, then $g(P^{(2)},1^{(n-2)})=<0>$.
 \end{theorem}
 \begin{proof}
  Let $g(P^{(2)},1^{(n-2)}) \neq <0>$. Suppose that $g(r_1^n) \in P$ for some $r_1^n \in R$ such that $g(r_i^{(2)},1^{(n-2)}) \notin P$. If $0 \neq g(r_1^n)$, then we have $g(r_1^{i-1},1,r_{i+1}^n) \in {\bf r^{(m,n)}}(P)$ as $P$ is an $n$-ary $wsq$-primary hyperideal of $R$. We assume that $ g(r_1^n)=0$. If $g(r_i,P,1^{(n-2)}) \neq <0>$, then we get $g(r_i,p,1^{(n-2)}) \neq 0$ for some $p \in P$ which implies 
  
  $\hspace{2cm}0 \neq g(r_i,f(g(r_1^{i-1},1,r_{i+1}^n),p,0^{(m-2)}),1^{(n-2)})$
  
  $\hspace{2.3cm}=f(g(r_1^n),g(r_i,p,1^{(n-2)}),0^{(m-2)})$
  
  $\hspace{2.3cm} \subseteq P$.\\ Since $P$ is an $n$-ary $wsq$-primary hyperideal of $R$ and $g(r_i^{(2)},1^{(n-2)}) \notin P$, we obtain $f(g(r_1^{i-1},1,r_{i+1}^n),p,0^{(m-2)}) \subseteq {\bf r^{(m,n)}}(P)$ and so $g(r_1^{i-1},1,r_{i+1}^n) \in {\bf r^{(m,n)}}(P)$. Let us assume $g(g(r_1^{i-1},1,r_{i+1}^n),P,1^{(n-2)}) \neq <0>$. Then $g(g(r_1^{i-1},1,r_{i+1}^n),p^{\prime},1^{(n-2)}) \neq 0$ for some $p^{\prime} \in P$. Therefore $0 \neq g(f(r_i,p^{\prime},0^{(m-2)}),g(r_1^{i-1},1,r_{i+1}^n),1^{(n-2)})
  =f(g(r_1^n),g(r_1^{i-1},p^{\prime},r_{i+1}^n),1^{(n-2)}),0^{(m-2)}) \subseteq P$. Since $P$ is an $n$-ary $wsq$-primary hyperideal of $R$ and $g(f(r_i,p^{\prime},0^{(m-2)})^{(2)},1^{(n-2)}) \nsubseteq P$, then $g(r_1^{i-1},1,r_{i+1}^n) \in {\bf r^{(m,n)}}(P)$. Now we assume that $g(r_i,P,1^{(n-2)})=g(g(r_1^{i-1},1,r_{i+1}^n),P,1^{(n-2)})=<0>$. From $g(P^{(2)},1^{(n-2)}) \neq <0>$, it follows that $g(a,b,1^{(n-2)}) \neq 0$ for some $a,b \in P$. So 
  
  $\hspace{1.3cm}0 \neq g(f(r_i,a,0^{(m-2)}),f(g(r_1^{i-1},1,r_{i+1}^n),b,0^{(m-2)}),1^{(n-2)})$
  
  $\hspace{1.6cm}=
  f(g(r_1^n),g(r_i,b,1^{(n-2)}),g(r_1^{i-1},a,r_{i+1}^n),g(a,b,1^{(n-2)}),0^{(m-4)})$

 $\hspace{1.6cm} \subseteq P$.\\ 
 If $g(f(r_i,a,0^{(m-2)})^{(2)},1^{(n-2)}) \subseteq P$, then we obtain
 
  $\hspace{1cm}f(g(r_i^{(2)},1^{(n-2)}),g(r_i,a,1^{(n-2)})^{(2)},g(a^{(2)},1^{(n-2)}),0^{(m-2)})\subseteq P$\\ which means $g(r_i^{(2)},1^{(n-2)}) \in P$, a contradiction, so $g(f(r_i,a,0^{(m-2)})^{(2)},1^{(n-2)}) \nsubseteq P$.
 Since $P$ is $n$-ary $wsq$-primary and $g(f(r_i,a,0^{(m-2)})^{(2)},1^{(n-2)}) \nsubseteq P$, we have $f(g(r_1^{i-1},1,r_{i+1}^n),b,0^{(m-2)}) \subseteq {\bf r^{(m,n)}}(P)$ which means $g(r_1^{i-1},1,r_{i+1}^n) \in {\bf r^{(m,n)}}(P)$. Hence $P$ is an $n$-ary $sq$-primary hyperideal of $R$ which is a contradiction. Thus  $g(P^{(2)},1^{(n-2)})=<0>$.
\end{proof}
As a consequence of the previous theorem we give the following explicit
result:
\begin{corollary} \label{radi}
Suppose that $P$ is an $n$-ary $wsq$-primary hyperideal of $R$ such that is not $sq$-primary. Then ${\bf r^{(m,n)}}(P)={\bf r^{(m,n)}}(0)$. 
\end{corollary}
The intersection of $n$-ary $wsq$-primary hyperideals is discussed in the next theorem.
\begin{theorem}
Let $\{P_i\}_{i \in I}$ be a family of $n$-ary $wsq$-primary hyperideals of $R$ such that are not $sq$-primary. Then $P=\cap_{i \in I }P_i$ is an $n$-ary $wsq$-primary hyperideal of $R$. 
\end{theorem}
\begin{proof}
Since $\{P_i\}_{i \in I}$ are a family of $n$-ary $wsq$-primary hyperideals of $R$ such that are not $sq$-primary, we conclude that  ${\bf r^{(m,n)}}(P)={\bf r^{(m,n)}}(\cap_{i \in I }P_i)=\cap_{i \in I }{\bf r^{(m,n)}}(P_i)={\bf r^{(m,n)}}(0)$,  by Corollary \ref{radi}. Now, assume that $0 \neq g(r_1^n) \in P$ for some $r_1^n \in R$ but $g(r_1^{i-1},1,r_{i+1}^n) \notin {\bf r^{(m,n)}}(P)$. Therefore we have $0 \neq g(r_1^n) \in P_i$ and $g(r_1^{i-1},1,r_{i+1}^n) \notin {\bf r^{(m,n)}}(P_i)={\bf r^{(m,n)}}(0)$ for every $i \in I$. Since $P_i$ is $n$-ary $wsq$-primary, we have $g(r_i^{(2)},1^{(n-2)}) \in P_i$ for all $i \in I$ which implies $g(r_i^{(2)},1^{(n-2)}) \in P$. Consequently, $P$ is an $n$-ary $wsq$-primary hyperideal of $R$. 
\end{proof}
A proper hyperideal $P$ of $R$ is called  $n$-ary weakly primary  provided that for $r_1^n \in R$, $0 \neq g(r_1^n) \in P$ implies $r_i \in P$ or $g(r_1^{i-1},1_R,r_{i+1}^n) \in {\bf r^{(m,n)}}(P)$ for some $1 \leq i \leq n$.
\begin{theorem}
Let $P$ and $Q$ be proper hyperideals of $R$ such that $P \subseteq Q$. If $P$ is an $n$-ary weakly primary hyperideal of $R$, then $g(P,Q,1^{(n-2)})$ is $n$-ary $wsq$-primary hyperideal of $R$.
\end{theorem}
\begin{proof}
Assume that $0 \neq g(r_1^n) \in g(P,Q,1^{(n-2)})$ for some $r_1^n \in R$. Since $P$ is an $n$-ary weakly primary hyperideal of $R$ and $g(P,Q,1^{(n-2)}) \subseteq P$, we get $r_i \in P$ or $g(r_1^{i-1},1,r_{i+1}^n) \in {\bf r^{(m,n)}}(P)$. From $P \subseteq Q$, it follows that $f(r_i^{(2)},1^{(n-2)}) \in g(P,Q,1^{(n-2)})$ or $g(r_1^{i-1},1,r_{i+1}^n) \in {\bf r^{(m,n)}}(P)={\bf r^{(m,n)}}(g(P,Q,1^{(n-2)}))$. Consequently, $g(P,Q,1^{(n-2)})$ is $n$-ary $wsq$-primary hyperideal of $R$.
\end{proof}
\begin{corollary}
Let  $P$ be an $n$-ary weakly primary hyperideal of $R$. Then $g(P^{(2)},1^{(n-2)})$ is $n$-ary $wsq$-primary hyperideal of $R$.
\end{corollary}
\begin{theorem}
Let $P$ be a proper hyperideal of $R$. Then 
 $P$ is an $n$-ary $wsq$-primary hyperideal if and only if 
for every $r \in R$,  $\langle r \rangle \subseteq P_r$ or $P_r \subseteq {\bf r^{(m,n)}}(P)$ or $P_r \subseteq A_r$ such that $P_r=\{a \in R \ \vert \ g(r,a,1^{(n-2)}) \in P\}$ and $A_r=\{a \in R \ \vert \ g(r,a,1^{(n-2)})=0\}$.
\end{theorem}
\begin{proof}
$( \Longrightarrow )$ Let $P$ is an $n$-ary $wsq$-primary hyperideal of $R$ and $r \in R$. If $g(r^{(2)},1^{(n-2)}) \in P$, then $\langle r \rangle \subseteq P_r$. Let us assume $g(r^{(2)},1^{(n-2)}) \notin P$. Take $a \in P_r$. So $g(r,a,1^{(n-2)}) \in P$. Let $0 \neq g(r,a,1^{(n-2)}).$ Since $P$ is an $n$-ary $wsq$-primary hyperideal and $g(r^{(2)},1^{(n-2)}) \notin P$, we have $a=g(a,1^{(n-1)}) \in {\bf r^{(m,n)}}(P)$ which implies $P_r \subseteq {\bf r^{(m,n)}}(P)$.  Let $0= g(r,a,1^{(n-2)}).$ Therefore  $a \in A_r$ which means $P_r \subseteq A_r$. 

$( \Longrightarrow )$ Let $0 \neq g(r_1^n) \in P$ for some $r_1^n \in R$ such that $g(r_i^{(2)},1^{(n-2)}) \notin P$ for some $1 \leq i \leq n$. By the hypothesis, we have $P_{r_i} \subseteq A_{r_i}$ or $P_{r_i} \subseteq {\bf r^{(m,n)}}(P)$. The former  case leads to a contradiction. In the latter case, we get $g(r_1^{i-1},1,r_{i+1}^n) \in P_{r_i} \subseteq {\bf r^{(m,n)}}(P)$, as needed.  
\end{proof}
 In the following, we consider the relationship between an $n$-ary $wsq$-primary hyperideal and its radical.
 \begin{theorem}
Assume that $R$ is a Krasner $(m,n)$-hyperring such that has no non-zero nilpotent elements. If $P$ is an $n$-ary $wsq$-primary hyperideal of $R$, then ${\bf r^{(m,n)}}(P)$ is an $n$-ary weakly prime hyperideal of $R$.
 \end{theorem}
 \begin{proof}
 Let $P$ is an $n$-ary $wsq$-primary hyperideal of $R$. Supoose that $0 \neq g(r_1^n) \in {\bf r^{(m,n)}}(P)$ for some $r_1^n \in R$. Then there exists $s \in \mathbb{N}$ such that if $s \leq n$, then $g(g(r_1^n)^{(s)},1^{(n-s)}) \in P$. Since $R$  has no non-zero nilpotent elements, we conclude that $0 \neq g(g(r_1^n)^{(s)},1^{(n-s)})$. Since $P$ is an $n$-ary $wsq$-primary hyperideal of $R$ and
 
  $\hspace{1.5cm}g(r_i^{(s)},g(r_1^{i-1},1,r_{i+1}^n)^{(s)},1^{(n-2s)})$

$\hspace{2cm}=g(r_i^{(s)},g(r_1^{i-1},1,r_{i+1}^n)^{(s)},g(1^{(n)}),1^{(n-2s-1)})$

$\hspace*{2cm}=g(g(r_i^{(s)},1^{(n-s)}),g(g(r_1^{i-1},1,r_{i+1}^n)^{(s)},1^{(n-s)}),1^{(n-2)}) $

$\hspace{2cm}\in P$, \\
we have $g(g(r_i^{(s)},1^{(n-s)})^{(2)},1^{(n-2)})=g(r_i^{(2s)},1^{(n-2s)}) \in P$ which means $r_i \in  {\bf r^{(m,n)}}(P)$ or $g(r_1^{i-1},1,r_{i+1}^n)^{(s)},1^{(n-s)}) \in {\bf r^{(m,n)}}(P)$ which implies  $g(r_1^{i-1},1,r_{i+1}^n) \in {\bf r^{(m,n)}}(P)$ for some $1 \leq i \leq n$. If $s=l(n-1)+1$, then we are done by a similar argument.  Thus ${\bf r^{(m,n)}}(P)$ is an $n$-ary weakly prime hyperideal of $R$.
 \end{proof}
Recall from \cite{d1} that  a mapping
$h : R_1 \longrightarrow R_2$ is called a homomorphism for some Krasner $(m, n)$-hyperrings $(R_1, f_1, g_1)$ and $(R_2, f_2, g_2)$ if for all $x^m _1 \in R_1$ and $y^n_ 1 \in R_1$ we have
\begin{itemize}
\item[\rm(i)]~$h(1_{R_1})=1_{R_2}$,
\item[\rm(ii)]~$h(f_1(x_1,..., x_m)) = f_2(h(x_1),...,h(x_m)),$
\item[\rm(iii)]~$h(g_1(y_1,..., y_n)) = g_2(h(y_1),...,h(y_n)). $
\end{itemize}
 \begin{theorem} \label{map} 
Assume that $(R_1,f_1,g_1)$ and $(R_2,f_2,g_2)$ are two commutative Krasner $(m,n)$-hyperrings and $ h:R_1 \longrightarrow R_2$ is a homomorphism. Then: 
\begin{itemize}
\item[\rm(1)]~ If $h$ is a monomorphism and $P_2$ is an $n$-ary $wsq$-primary hyperideal of $R_2$, then $h^{-1} (P_2)$ is an $n$-ary $wsq$-primary hyperideal of $R_1$.
\item[\rm(2)]~ If $h$ is an epimorphism and $P_1$ is an $n$-ary $wsq$-primary hyperideal  of $R_1$ with $Ker(h) \subseteq P_1$, then $h(P_1)$ is an $n$-ary $wsq$-primary hyperideal of $R_2$.
\end{itemize}
\end{theorem}
\begin{proof}
$(1)$ Suppose that  $0 \neq g_1(r_1^n) \in h^{-1} (P_2)$ for some $r_1^n \in R_1$. Since $h$ is a monomorphism, we conclude that  $0 \neq h(g_1(r_1^n))=g_2(h(r_1), \cdots,h(r_n)) \in P_2$. Since $P_2$ is an $n$-ary $wsq$-primary hyperideal of $R_2$, then 

$g_2(h(r_i)^{(2)},1_{R_2}^{(n-2)})=h(g_1(r_i^{(2)},1_{R_1}^{(n-2)})) \in P_2$ \\or 

$g_2(h(r_1),\cdots,h(r_{i-1}),1_{R_2},h(r_{i+1}),\cdots,h(r_n))=h(g_1(r_1^{i-1},1_{R_1},r_{i+1}^n)) \in {\bf r^{(m,n)}}(P_2)$ for some $1 \leq i \leq n$. This means that $g_1(r_i^2,1_{R_1}^{(n-2)}) \in h^{-1}(P_2)$ or $g(r_1^{i-1},1_{R_1},r_{i+1}^n) \in h^{-1}({\bf r^{(m,n)}}(P_2))={\bf r^{(m,n)}}(h^{-1}(P_2))$. Consequently,  $h^{-1}(P_2)$ is an $n$-ary $wsq$-primary hyperideal of $R_1$.

$(2)$ Let  $0 \neq g_2(t_1^n) \in h(P_1)$ for some $t_1^n \in R_2$. Since $h$ is an epimorphism, then there exist $r_1^n \in R_1$ with $h(r_1)=t_1,...,h(r_n)=t_n$. Therefore
$0 \neq h(g_1(r_1^n))=g_2(h(r_1),...,h(r_n))=g_2(t_1^n) \in h(P_1)$.
Since $Ker(h) \subseteq P_1$, then we have $0 \neq g_1(r_1^n) \in P_1$. Since $P_1$ is an $n$-ary $wsq$-primary hyperideal  of $R_1$, we obtain $g_1(r_i^{(2)},1_{R_1}^{(n-2)}) \in P_1$ or $g_1(r_1^{i-1},1_{R_1},r_{i+1}^n) \in {\bf r^{(m,n)}}(P_1) $. Therefore 

$\hspace{1.5cm}h(g_1(r_i^{(2)},1_{R_1}^{(n-2)}))=g_2(h(r_i)^{(2)},1_{R_2}^{(n-2)})=g_2
(t_i^{(2)},1_{R_2}^{(n-2)}) \in h(P_1)$\\
 or
 
  $\hspace{1cm}h(g_1(r_1^{i-1},1_{R_1},r_{i+1}^n))=g_2(h(r_1),\cdots,h(r_{i-1}),1_{R_2},h(r_{i+1}),\cdots,h(r_n))$
  
  $\hspace{4.3cm}=g_2(t_1^{i-1},1_{R_2},t_{i+1}^n)$
  
  $\hspace{4.3cm} \in  h({\bf r^{(m,n)}}(P_1))$
  
  $\hspace{4.3cm} \subseteq {\bf r^{(m,n)}}(h(P_1))$\\
Thus $h(P_1)$ is an $n$-ary $wsq$-primary hyperideal of $R_2$.
\end{proof}
\begin{corollary} 
Let  $P$ and $Q$ be two proper hyperideals of  $R$ with $Q \subseteq P$. 
\begin{itemize}
\item[\rm(1)]~ If $P$ is an $n$-ary $wsq$-primary hyperideal of $R$, then $P/Q$ is an $n$-ary $wsq$-primary hyperideal of $R/Q$. 
\item[\rm(2)]~ If $Q$ is an $n$-ary $wsq$-primary hyperideal of $R$ and $P/Q$ is an $n$-ary $wsq$-primary hyperideal of $R/Q$, then $P$ is an $n$-ary $wsq$-primary hyperideal of $R$.
\end{itemize}
\end{corollary}
\begin{proof}
$(1)$ Consider the  epimorphism $\pi: R \longrightarrow R/Q$, defined by $r \longrightarrow f(r,Q,0^{(m-2)})$. Now, the claim follows by using Theorem \ref{map} (2). 

$(2)$ Assume that $0 \neq g(r_1^n) \in P$ for $r_1^n \in R$. If $0 \neq g(r_1^n) \in Q$, then $g(r_i^{(2)},1^{(n-2)}) \in Q \subseteq P$ or $g(r_1^{i-1},1,r_{i+1}^n) \in {\bf r^{(m,n)}} (Q) \subseteq {\bf r^{(m,n)}} (P)$ for some $1 \leq i \leq n$ as $Q$ is an $n$-ary $wsq$-primary hyperideal of $R$. If $0 \neq g(r_1^n) \notin Q$, then we conclude that 

 $ 0 \neq f(g(r_1^n),Q,0^{(m-2)})=g(f(r_1,Q,0^{(m-2)}),\cdots,f(r_n,Q,0^{(m-2)})) \in P/Q.$\\
 Since $P/Q$ is an $n$-ary $wsq$-primary hyperideal of $R/Q$, we get 
 
$g(f(r_i,Q,0^{(m-2)})^{(2)},f(1,Q,0^{(m-2)})^{(n-2)})=
f(g(r_i^{(2)},1^{(n-2)}),Q,0^{(m-2)}) \in P/Q$\\ or 
   
   $g(f(r_1,Q,0^{(m-2)}),\cdots,f(r_{i-1},Q,0^{(m-2)}),f(1,Q,0^{(m-2)}),f(r_{i+1},Q,0^{(m-2)}),$
   
   $\hspace{2cm}\cdots,f(r_n,Q,0^{(m-2)}))$
   
   $\hspace{1cm}=f(g(r_1^{i-1},1,r_{i+1}^n),Q,0^{(m-2)})$
   
   $ \hspace{1cm}\in {\bf r^{(m,n)}}(P/Q)$
   
   $\hspace{1cm} = {\bf r^{(m,n)}}(P)/Q$.\\
   This means that $g(r_i^{(2)},1^{(n-2)}) \in P$ or $g(r_1^{i-1},1,r_{i+1}^n) \in {\bf r^{(m,n)}}(P)$. This shows that $P$ is an $n$-ary $wsq$-primary hyperideal of $R$.
\end{proof}
\begin{theorem}
Let  $H$ be a subhyperring of $R$ and  $P$  be a proper hyperideals of  $R$ such that  $H \nsubseteq P$. If $P$ is  an $n$-ary $wsq$-primary hyperideal of $R$, then $H \cap P$ is an $n$-ary $wsq$-primary hyperideal of $H$.
\end{theorem}
\begin{proof}
Applying Theorem \ref{map} (1) to the injection $j:H \longrightarrow R$, defined by $j(x)=x$ for all $x \in H$, we conclude that $j^{-1}(P)=H \cap P$ is an $n$-ary $wsq$-primary hyperideal of $H$.
\end{proof}
Recall from \cite{sorc1} that   a non-empty subset $S$ of a Krasner $(m,n)$-hyperring $R$  is called an $n$-ary multiplicative subset if  $g(s_1^n) \in S$ for $s_1,...,s_n \in S$. The concept of Krasner $(m,n)$-hyperring of fractions was introduced in \cite{mah5}.
\begin{theorem} 
Assume that $R$ is a Krasner $(m,n)$-hyperring and $S$ is an $n$-ary multiplicative subset of $R$ such that $1 \in S$. If $P$ is an $n$-ary $wsq$-primary hyperideal of $R$ such that $P\cap S=\varnothing$, then $S^{-1}P$ is an $n$-ary $wsq$-primary hyperideal of $S^{-1}R$.
\end{theorem}
\begin{proof}
Let $\frac{r_1}{s_1},...,\frac{r_n}{s_n} \in S^{-1}R$ with $0 \neq G(\frac{r_1}{s_1},...,\frac{r_n}{s_n}) \in S^{-1}P$ .
Therefore  $\frac{g(r_1^n)}{g(s_1^n)} \in S^{-1}P$. It means that there exists $v \in S$ such that $0 \neq g(v,g(r_1^n),1^{(n-2)}) \in P$ and so $g(r_1^{i-1},g(v,r_i,1^{(n-2)}),r_{i+1}^n) \in I$. Without destroying the generality, we may assume that $g(r_1^{n-1},g(v,r_n,1^{(n-2)})) \in P$. Since $P$ is an $n$-ary $wsq$-primary hyperideal of $R$, then at least one of the cases holds: $g(a_i^{(2)},1^{(n-2)}) \in P$ for some $1 \leq i \leq n-1$, $g(g(v,r_n,1^{(n-2)})^{(2)},1^{(n-2)}) \in P$, $g(r_1^{i-1},1,r_{i+1}^{n-1},g(v,r_n,1^{(n-2)})) \in {\bf r^{(m,n)}}(P)$ for some $1 \leq i \leq n-1$ or $g(r_1^{n-1},1) \in {\bf r^{(m,n)}}(P)$.\\ 
If $g(r_i^{(2)},1^{(n-2)}) \in P$ for some $1 \leq i \leq n-1$, then $G(\frac{r_i}{s_i}^{(2)},\frac{1}{1}^{(n-2)})=
\frac{g(r_i^{(2)},1^{(n-2)})}{g(s_i^{(2)},1^{(n-2)})} \in S^{-1}P$. 
If $g(g(v,r_n,1^{(n-2)})^{(2)},1^{(n-2)}) \in P$ then $G(\frac{r_n}{s_n}^{(2)},\frac{1}{1}^{(n-2)})=
\frac{g(r_n^{(2)},1^{(n-2)})}{g(s_n^{(2)},1^{(n-2)})}=\frac{g(g(v^{(2)},1^{(n-2)}), r_n^{(2)},1^{(n-3)})}{g(g(v^{(2)},1^{(n-2)}),s_n^{(2)},1^{(n-3)})}=\frac{g(g(v,r_n,1^{(n-2)})^{(2)},1^{(n-2)})}{g(g(v,s_n,1^{(n-2)})^{(2)},1^{(n-2)})} \in S^{-1}P$. \\If $g(r_1^{i-1},1,r_{i+1}^{n-1}, g(v,r_n,1^{(n-2)})) \in {\bf r^{(m,n)}}(P)$ for some $1 \leq i \leq n-1$, then  $G(\frac{r_1}{s_1},...,\frac{r_{i-1}}{s_{i-1}},\frac{1}{1},\frac{r_{i+1}}{s_{i+1}},...,\frac{r_{n-1}}{s_{n-1}},\frac{r_n}{s_n})=  \frac{g(r_1^{i-1},v,a_{i+1}^{n-1},r_n)}{g(s_1^{i-1},v,s_{i+1}^{n-1},s_n)}=\frac{g(r_1^{i-1},1,r_{i+1}^{n-1},g(v,r_n,1^{(n-2)}))}{g(s_1^{i-1},1,s_{i+1}^{n-1},g(v,s_n,1^{(n-2)}))} \in S^{-1}({\bf r^{(m,n)}}(P))={\bf r^{(m,n)}}(S^{-1}P)$, by Lemma 4.7 in \cite{mah5}. If $g(r_1^{n-1},1) \in {\bf r^{(m,n)}}(P)$, then $G(\frac{r_1}{s_1},...,\frac{r_{n-1}}{s_{n-1}},\frac{1}{1})=\frac{g(r_1^{n-1},1)}{g(s_1^{n-1},1)} \in S^{-1}({\bf r^{(m,n)}}(P))={\bf r^{(m,n)}}(S^{-1}P)$. Consequently,  $S^{-1}P$ is an $n$-ary $wsq$-primary hyperideal of $S^{-1}R$.
\end{proof}
\begin{theorem} \label{ali3}
Let  $(R_1, f_1, g_1)$ and $(R_2, f_2, g_2)$ be two Krasner $(m,n)$-hyperrings with  scalar identitis $1_{R_1}$ and $1_{R_2}$. Assume that $P_1$ is a proper hyperideal of $R_1$. Then the
followings  are equivalent:
\begin{itemize}
\item[\rm(i)]~ $P_1 \times R_2$ is an $n$-ary $wsq$-primary hyperideal of $R_1 \times R_2$.
\item[\rm(ii)]~$P_1 \times R_2$ is an $n$-ary $sq$-primary hyperideal of $R_1 \times R_2$.
\item[\rm(iii)]~$P_1$ is an $n$-ary $sq$-primary hyperideal of $R_1$.
\end{itemize}
\end{theorem}
\begin{proof}
(i)$ \Longrightarrow $(ii) Let $P_1 \times R_2$ be an $n$-ary $wsq$-primary hyperideal of $R_1 \times R_2$. From $P_1 \times R_2 \nsubseteq {\bf r^{(m,n)}}(0)$, it follows that $P_1 \times R_2$ is an $n$-ary $sq$-primary hyperideal of $R_1 \times R_2$ by Corollary  \ref{radi}.

(ii)$ \Longrightarrow $(iii) Suppose that  $g_1(r_1^n) \in P_1$ for some $r_1^n \in R_1$. Then we have $g_1 \times g_2((r_1,1_{R_2}), \cdots, (r_n,1_{R_2})) =(g_1(r_1^n),1_{R_2}) \in P_1 \times R_2$. Since $P_1 \times R_2$ is an $n$-ary $sq$-primary hyperideal of $R_1 \times R_2$, we get $g_1 \times g_2((r_i,1_{R_2}),1_{R_2})^{(2)},(1_{R_1},1_{R_2})^{(n-2)})=(g_1(r_i^{2},1_{R_1}^{(n-2)}),1_{R_2}) \in P_1 \times R_2$ or 

$g_1 \times g_2((r_1,1_{R_2}),\cdots,(r_{i-1},1_{R_2}),(1_{R_1},1_{R_2}),(r_{i+1},1_{R_2}),\cdots,(r_n,1_{R_2}))$

$\hspace{1cm}=(g_1(r_1^{i-1},1_{R_1},r_{i+1}^n),1_{R_2})$

$\hspace{1cm} \in {\bf r^{(m,n)}}(P_1 \times R_2)$

$\hspace{1cm}={\bf r^{(m,n)}}(P_1) \times R_2$ \\ for some $1 \leq i \leq n$. This implies that $g_1(r_i^{(2)},1_{R_1}^{(n-2)}) \in P_1$ or $g_1(r_1^{i-1},1_{R_1},r_{i+1}^n) \in {\bf r^{(m,n)}}(P_1)$. Hence  $P_1$ is an $n$-ary $sq$-primary hyperideal of $R_1$.

(iii)$ \Longrightarrow $(i) It follows from Theorem \ref{ali2}.
\end{proof}
\begin{corollary}
Let  $(R_1, f_1, g_1)$ and $(R_2, f_2, g_2)$ be two Krasner $(m,n)$-hyperrings with  scalar identitis $1_{R_1}$ and $1_{R_2}$. Assume that $P_1$ is a proper hyperideal of $R_1$ and $P_2$ is a proper hyperideal of $R_2$ with $ P_1 \times P_2 \neq <0>$. Then the
following statements  are equivalent:
\begin{itemize}
\item[\rm(i)]~ $P_1 \times P_2$ is an $n$-ary $wsq$-primary hyperideal of $R_1 \times R_2$.
\item[\rm(ii)]~$P_1$ is an  $n$-ary $sq$-primary hyperideal   of $R_1$ and $P_2=R_2$ or $P_2$ is an  $n$-ary $sq$-primary hyperideal of  $R_2$ and $P_1=R_1$.
\item[\rm(iii)]~$P_1 \times P_2$ is an $n$-ary $sq$-primary hyperideal of $R_1 \times R_2$.
\end{itemize}
\end{corollary}
\begin{proof}
(i) $\Longrightarrow$ (ii) Let $P_1 \times P_2 \neq <0>$ be an $n$-ary $wsq$-primary hyperideal of $R_1 \times R_2$. We assume that $P_2 \neq <0>$. Take any $0 \neq p_2 \in P_2$. Then $g_1 \times g_2((1_{R_1},p_2),(0,1_{R_2})^{(n-1)})=(0,p_2) \in P_1 \times P_2$. Since $P_1 \times P_2$ is an $n$-ary $wsq$-primary hyperideal of $R_1 \times R_2$ and $(0,0) \neq g_1 \times g_2((1_{R_1},p_2),(0,1_{R_2})^{(n-1)}) \in P_1 \times P_2$, we have  $g_1 \times g_2((1_{R_1},p_2)^{(2)},(1_{R_1},1_{R_2})^{(n-2)})=(1_{R_1},g_2(p_2^{(2)},1_{R_2}^{(n-2)})) \in P_1 \times P_2$ or $g_1 \times g_2((1_{R_1},1_{R_2}),(0,1_{R_2})^{(n-1)})=(0,1_{R_2}) \in {\bf r^{(m,n)}}(P_1 \times P_2)={\bf r^{(m,n)}}(P_1 ) \times {\bf r^{(m,n)}}(P_2)$. Then we conclude that $1_{R_1} \in P_1$ or $1_{R_2} \in P_2$ which means $P_1 =R_1$ or $P_2 = R_2$. If $P_1=R_1$, then $P_2$  is an  $n$-ary $sq$-primary hyperideal of  $R_2$ by Theorem \ref{ali3}. Similiar for the other case.

(ii) $\Longrightarrow$ (iii) It follows from Theorem \ref{ali2}.

(iii) $\Longrightarrow$ (i) Since every $n$-ary $sq$-primary hyperideal is $n$-ary $wsq$-primary, we are done.
\end{proof}
\section{Conclusion}
This paper included the structures of $n$-ary $q$-primary, $(k,n)$-absorbing $q$-primary, $sq$-primary and $wsq$-primay hyperideals of a Krasner $(m,n)$-hyperring $R$. Several important results in these classes of hyperideals were discussed and proved. The relationship of them is examined. Moreover,  the stabilty of the notions were studied in some hyperring-theoretic constructions.
 Based on our work, we propose some open problems to researchers:
 \begin{itemize}
\item[\rm(1)]~To introduce and study $(k,n)$-absorbing $sq$-primary hyperideals. 
\item[\rm(2)]~  To introduce and study $(k,n)$-absorbing $wsq$-primary hyperideals. 
\item[\rm(3)]~ To introduce and study  $wsq$-$\delta$-primary hyperideals where $\delta$ is a hyperideal expansion of a Krasner $(m,n)$-hyperring $R$. 
\end{itemize}

\end{document}